\documentclass[12pt,leqno,a4paper]{amsart}
\usepackage{amssymb,enumerate}
\newcommand{\FF}{{\mathbb{F}}}

\newcommand{\fA}{{\mathfrak{A}}}
\newcommand{\fS}{{\mathfrak{S}}}

\newcommand{\bC}{{\mathbf{C}}}
\newcommand{\bG}{{\mathbf{G}}}
\newcommand{\bH}{{\mathbf{H}}}
\newcommand{\bL}{{\mathbf{L}}}
\newcommand{\bM}{{\mathbf{M}}}
\newcommand{\bT}{{\mathbf{T}}}
\newcommand{\bX}{{\mathbf{X}}}
\newcommand{\bY}{{\mathbf{Y}}}

\newcommand{\cE}{{\mathcal{E}}}

\newcommand{\diag}{{\operatorname{diag}}}
\newcommand{\Irr}{{\operatorname{Irr}}}

\newcommand{\SL}{{\operatorname{SL}}}
\newcommand{\PGL}{{\operatorname{PGL}}}
\newcommand{\PSL}{{\operatorname{L}}}
\newcommand{\PSp}{{\operatorname{S}}}
\newcommand{\SU}{{\operatorname{SU}}}

\newcommand{\PSU}{{\operatorname{U}}}
\newcommand{\GL}{{\operatorname{GL}}}
\newcommand\RLG{{R_\bL^\bG}}
\newcommand\RTG{{R_\bT^\bG}}

\newcommand{\tw}[1]{{}^#1\!}

\let\la=\lambda

\newtheorem{thm}{Theorem}[section]
\newtheorem{lem}[thm]{Lemma}

\newtheorem{prop}[thm]{Proposition}

\newtheorem*{thmA}{Main Theorem}

\theoremstyle{definition}

\theoremstyle{remark}

\begin{document}

\title[Brauer's height zero conjecture for quasi-simple groups]{Brauer's height zero conjecture\\ for quasi-simple groups}

\date{\today}

\author{Radha Kessar}
\address{Department of Mathematical Sciences, City University London,
         Northampton Square, London EC1V 01B, U.K.}
\email{Radha.Kessar.1@city.ac.uk}
\author{Gunter Malle}
\address{FB Mathematik, TU Kaiserslautern, Postfach 3049,
         67653 Kaisers\-lautern, Germany.}
\email{malle@mathematik.uni-kl.de}

\thanks{The second author gratefully acknowledges financial support by ERC
  Advanced Grant 291512.}



\dedicatory{To the memory of   Sandy Green}
\begin{abstract}
We show that Brauer's height zero conjecture holds for blocks of finite
quasi-simple groups. This result is used in Navarro--Sp\"ath's reduction of
this conjecture for general groups to the inductive Alperin--McKay condition
for simple groups.
\end{abstract}

\maketitle


\section{Introduction }   \label{sec:intro}

In this paper we verify that the open direction of Richard Brauer's 1955 height
zero conjecture (BHZ) holds for blocks of finite quasi-simple groups:

\begin{thmA}   \label{thm:A}
 Let $S$ be a finite quasi-simple group, $\ell$ a prime and $B$ an $\ell$-block
 of $S$. Then $B$ has abelian defect groups if and only if all
 $\chi\in\Irr(B)$ have height zero.
\end{thmA}

The proof of one direction of Brauer's height zero conjecture, that blocks
with abelian defect groups only contain characters of height zero, was
completed in \cite{KM}. Subsequently it was shown by Gabriel Navarro and
Britta Sp\"ath \cite{NS14} that the other direction of (BHZ) can be reduced
to proving the following for all finite quasi-simple groups $S$:
\begin{enumerate}
 \item[(1)] (BHZ) holds for $S$, and
 \item[(2)]  the inductive form of the Alperin--McKay conjecture holds for
  $S/Z(S)$.
\end{enumerate}
Here, we show that the first statement holds. The main case, when $S$ is simple
of Lie type, is treated in Section~\ref{sec:height0-Lie}, and then the proof
of the Main Theorem is completed in Section~\ref{sec:height0}.

\section{Brauer's height zero conjecture for groups of Lie type}  \label{sec:height0-Lie}

In this section we show that (BHZ) holds for quasi-simple groups of Lie type.
This constitutes the central part of the proof of our Main Theorem.

Throughout, we work with the following setting. We let $\bG$ be a connected
reductive linear algebraic group over an algebraic closure of a finite field
of characteristic~$p$, and $F:\bG\rightarrow\bG$ a Steinberg endomorphism
with finite group of fixed points $\bG^F$. It is well-known that apart
from finitely many exceptions, all finite  quasi-simple groups of Lie type can
be obtained as $\bG^F/Z$ for some central subgroup $Z\le \bG^F$ by choosing
$\bG$ simple of simply connected type.

We let $\bG^*$ be dual to $\bG$, with  compatible Steinberg endomorphism
again denoted $F$. Recall that by the results of Lusztig the set $\Irr(\bG^F)$
of complex irreducible characters of $\bG^F$ is a disjoint union of rational
Lusztig series $\cE(\bG^F,s)$, where $s$ runs over the semisimple elements
of $\bG^{*F}$ up to conjugation.

\subsection{Groups of Lie type in their defining characteristic}
We first consider the easier case of groups of Lie type in their defining
characteristic, where we need the following:

\begin{lem}   \label{lem:root}
 Let $\bG$ be simple of adjoint type, not of type $A_1$, with Frobenius
 endomorphism $F:\bG\rightarrow\bG$. Then every coset of $[\bG^F,\bG^F]$ in
 $\bG^F$ contains a (semisimple) element centralising a root subgroup of
 $\bG^F$.
\end{lem}

\begin{proof}
First note that by inspection any of the rank~2 groups $\PSL_3(q)$,
$\PSU_3(q)$, and $\PSp_4(q)$ (and hence also $\PSU_4(q)$) contains a root
subgroup $U\cong\FF_q^+$ all of
whose non-identity elements are conjugate under a maximally split torus. Now
if $\bG$ is not of type $A_1$ with $[\bG^F,\bG^F]<\bG^F$ then it contains an
$F$-stable Levi subgroup $\bH$ of type $A_2$, $B_2$, or $A_3$,
and thus $\bG^F$ contains a root subgroup $U$ all of whose
non-trivial elements are conjugate under the maximally split torus of
$[\bH^F,\bH^F]\le[\bG^F,\bG^F]$. But $\bG^F=[\bG^F,\bG^F]\bT^F$ for any
$F$-stable maximal torus $\bT$ of $\bG$ (see \cite[Ex.~30.13]{MT}). Thus any
coset of $[\bG^F,\bG^F]$ in $\bG^F$ contains semisimple elements which
centralise $U$.
\end{proof}

\begin{prop}   \label{prop:defchar} 
 Let $\bG$ be simple, simply connected, not of type $A_1$, and $Z\le \bG^F$ be
 a central subgroup such that $S=\bG^F/Z$ is quasi-simple of Lie type in
 characteristic~$p$. Then any $p$-block of $S$ of positive defect contains
 characters of positive height.
\end{prop}

\begin{proof}
By assumption, $S/Z(S)\not\cong\PSL_2(q)$. By the result of Humphreys
\cite{Hum}, the $p$-blocks of $\bG^F$ of positive defect are in bijection with
$\Irr(Z(\bG^F))$ and are of full defect. The principal block of $\bG^F$
contains all the unipotent characters of $\bG^F$, hence a character of
positive height e.g.~by \cite[Thm.~6.8]{MaH} (except when $S=\PSp_4(2)=\fS_6$
where the statement can be checked directly).
\par
Now assume that $Z(\bG^F)\ne1$, and $B$ is the $p$-block of $\bG^F$ lying over
the non-trivial character $\la\in\Irr(Z(\bG^F))$. By the work of Lusztig
\cite{Lu88} there is a natural isomorphism
$\Irr(Z(\bG^F))\cong \bG^{*F}/[\bG^{*F},\bG^{*F}]$ such that for any $s$ in the
coset corresponding to $\la$ all characters of $\cE(\bG^F,s)$ lie over $\la$,
hence in $B$.
Now by Lemma~\ref{lem:root} this coset contains a semisimple element $s_\la$
centralising a root subgroup of $\bG^{*F}$. Then $C_{\bG^{*F}}(s_\la)$
contains a root subgroup, hence has semisimple rank at least~1. By Lusztig's
Jordan decomposition of characters, the regular character in $\cE(\bG^F,s_\la)$
corresponds to the Steinberg character of $C_{\bG^{*F}}(s_\la)$, so has
positive $p$-height, and it lies in $B$.
\end{proof}

\subsection{Unipotent pairs and $e$-cuspidality}
We now turn to the investigation of $\ell$-blocks for primes $\ell\ne p$, which
is considerably more involved.
For the rest of this section we assume that $F: \bG\to\bG$ is a Frobenius
morphism with respect to some $\FF_q$-structure on $\bG$. Let $\ell$ be a prime
not dividing $q$ and let $e=e_\ell(q)$, where $e_\ell(q)$ is the order of
$q$ modulo $\ell$ if $\ell$ is odd and is the order of $q$ modulo $4$ if
$\ell =2$.

By a unipotent pair for $\bG^F$ we mean a pair $(\bL,\la)$, where $\bL$ is an
$F$-stable Levi subgroup of $\bG$ and $\la\in\cE(\bL^F,1)$. If $\bL$ is
$d$-split in $\bG$, then $(\bL,\la)$ is  said to be a unipotent $d$-pair and
if in addition $\la$ is a  unipotent $d$-cuspidal character of $\bL^F$, then
$(\bL, \la)$ is said to be a unipotent $d$-cuspidal pair.

Recall that if $\bL$ is an $F$-stable Levi subgroup of $\bG $, then
$\bar \bL := \bL/Z(\bG)$ is an $F$-stable Levi subgroup of $\bG/Z(\bG)$ and
$ \bL_0 := \bL\cap [\bG, \bG]$ is an $F$-stable Levi subgroup of $[\bG, \bG]$;
the maps $ \bL\mapsto\bar\bL$  and $\bL\mapsto\bL_0$ give bijections between
the sets of $F$-stable Levi subgroups of $\bG$ and of $\bG/Z(\bG)$ and between
the sets of $F$-stable Levi subgroups of $\bG$ and of $[\bG,\bG]$. Also
recall that the natural maps $\bL\to\bL/Z(\bG)$ and $\bL\cap[\bG,\bG]\to\bL$
induce  degree preserving bijections between $\cE(\bL^F,1)$, $\cE(\bar\bL^F,1)$
and $\cE(\bL_0^F,1)$. Hence there are natural bijections between the sets of
unipotent pairs of $\bG^F $, $(\bG/Z(\bG))^F$ and of $[\bG,\bG]^F$ and these
preserve the  properties of being $d$-split and of being $d$-cuspidal
(see \cite[Sec.~3]{CE94}).

\begin{lem}   \label{lem:redweyl}
 Let $(\bL,\la)$, $(\bL_0,\la_0)$ and $(\bar\bL, \bar\la)$ be corresponding
 unipotent pairs for $\bG^F$, $[\bG, \bG]^F$ and $(\bG/Z(\bG))^F$. Then,
 $$W_{[\bG,\bG]^F}(\bL_0, \la_0)\cong W_{\bG^F}(\bL, \la) \cong
    W_{(\bG/Z(\bG))^F}(\bar \bL, \bar \la).$$
\end{lem}

\begin{proof}
Let $\bar \bG = \bG/Z(\bG)$. The canonical map $\bG \to\bar\bG $ induces an
injective map from $W_{\bG^F}( \bL, \la)$ into $W_{\bar\bG^F}(\bar\bL,\bar\la)$.
Conversely, let $x \in \bG $ be such that its image $\bar x\in\bar\bG$ is in
$N_{\bar\bG^F}(\bar\bL,\bar\la)$. Then $x$ normalises $\bL$ as well as
$\bL^F$ and stablises $\la$. Further, by the Lang--Steinberg theorem,
$xt\in\bG^F$ for some $t$ lying in an $F$-stable maximal torus $\bT$ of $\bL$.
Since $N_\bT(\bL^F)$ stabilises $\la$, we have that $xt\in N_{\bG^F}(\bL,\la)$.
Further, since $\bar x\in\bar\bG^F$, $\bar t\in\bar\bL^F$, and hence
$xt\bL^F\mapsto\bar x\bar\bL^F$ under the inclusion of
$W_{\bG^F}(\bL,\la)$ in $W_{\bar\bG^F}(\bar \bL,\bar\la)$. The proof for
the isomorphism
$$ W_{[\bG,\bG]^F}(\bL_0, \la_0) \cong W_{\bG^F}(\bL, \la)$$
is similar.
\end{proof}

Next, we note the following consequence of \cite[Prop.~1.3]{CE94}.

\begin{lem}   \label{lem:redecusp}
 Suppose that $\bG=[\bG, \bG]$ is simply connected. Let $\bG_1,\ldots,\bG_r$
 be a set of representatives for the $F$-orbits on the set of simple
 components of $\bG $ and for each $i$ let $d_i$ denote the length of the
 $F$-orbit of $\bG_i$. For a Levi subgroup $\bL$ of $\bG$, let
 $\bL_i =\bL \cap \bG_i$.
 Then $\bL $ is $F$-stable if and only if $\bL_i $ is $F^{d_i} $-stable
 for all $i$ and in this case
 $$\bL = (\bL_1F(\bL_1)\cdots F^{d_1-1}(\bL_1))\cdots(\bL_r F(\bL_r)\cdots
    F^{d_r-1} (\bL_r)). $$
 Further, projecting onto the $\bG_i$ component in each $F$-orbit induces an
 isomorphism
 $$\bL^F \cong \bL_1^{F^{d_1}}\times \cdots \times \bL_r^{F^{d_r}}.$$
 If, under the above isomorphism, $\la\in\cE(\bL^F,1)$ corresponds to
 $\la_1\times \cdots \times \la_r$, with $\la_i\in\cE(\bL^{F^{d_i}},1)$,
 then $(\bL,\la)$ is an $e$-cuspidal pair for $\bG^F$ if and only if
 $(\bL_i^{F^{d_i}},\la_i)$ is  an $e_\ell(q^{d_i})$-cuspidal pair for
 $\bG_i^{F^{d_i}}$ for each $i$.
\end{lem}

\begin{lem}   \label{lem:owncentraliser}
 Suppose that either $\ell$ is odd or that $\bG$ has no components of
 classical type $A,B,C$, or $D$. Let $(\bL,\la)$ be a unipotent
 $e$-cuspidal pair of $\bG^F$. Then, $\bL =C_\bG^\circ(Z(\bL)^F_\ell)$.
\end{lem}

\begin{proof}
We claim that it suffices to prove the result in the case that $\bG $ is
semisimple. Indeed, let $\bG_0 =[\bG, \bG]$, $\bL_0 =\bL \cap\bG_0$ and
$\la_0$ be the restriction of $\la$ to $\bL_0 ^F$. Then, $(\bL_0,\la_0)$
is a unipotent $e$-cuspidal pair of $\bG_0^F$. Suppose that
$\bL_0 =C_{\bG_0}^\circ (Z(\bL_0)^F_\ell)$. Since $\bG = Z^\circ(\bG)\bG_0 $,
we have that
$$C_\bG(Z(\bL_0)^F_\ell) =Z^\circ(\bG) C_{\bG_0}(Z(\bL_0)^F_\ell),$$
hence
$$ C_\bG ^\circ( Z(\bL_0)^F_\ell) =Z^\circ(\bG)C_{\bG_0}^\circ
  (Z(\bL_0)^F_\ell) = Z^\circ (\bG)\bL_0 = \bL.$$
Here the first equality holds since
$C_\bG ( Z(\bL_0)^F_\ell) /Z^\circ(\bG) C_{\bG_0}^\circ(Z(\bL_0)^F_\ell)$
is a surjective image of $C_{\bG_0}(Z(\bL_0)^F_\ell)/
C_{\bG_0}^\circ(Z(\bL_0)^F_\ell)$ and hence is finite.
On the other hand, we have that $Z(\bL_0)^F_\ell \leq Z(\bL)^F_\ell$
whence $C_\bG (Z(\bL)^F_\ell)\leq C_\bG(Z(\bL_0)^F_\ell)$
and the claim follows.

We assume from now on that $\bG =[\bG,\bG]$. We claim that it suffices
to prove the result in the case that $\bG$ is simply connected. Indeed, let
$\hat \bG\to\bG$ be an $F$-compatible simply connected covering of $\bG$,
with finite central kernel, say $Z$. Let $\hat\bL$ be the inverse image of
$\bL$ in $\hat\bG$ and let $\hat \la_0\in\Irr(\hat \bL^F)$ be the (unipotent)
inflation of $\la$. By Lemma~\ref{lem:redweyl} $(\hat \bL, \hat \la)$ is an
$e$-cuspidal unipotent pair of $\hat \bL ^F $. Let $\hat A= Z(\hat\bL)^F_\ell$
and suppose that $C_{\hat\bG}^\circ(\hat A)= \hat \bL$.
Let $A =\hat AZ/Z$ and let $\bC$ be the inverse image in $\hat\bG$ of
$C_\bG(A)$. Then $C_{\hat\bG}(\hat A)= C_{\hat \bG}(\hat AZ)$ is a normal
subgroup of $\bC$ and $\bC/C_{\hat \bG}(\hat A)$ is isomorphic to a subgroup of
the automorphism group of $\hat AZ$. Since $\hat AZ$ is finite, it follows that
$\bC/C_{\hat \bG}(\hat A)$ is finite and hence $C_{\hat\bG}(\hat A)/Z$ has
finite index in $\bC/Z = C_{\bG}(\hat A)$. On the other hand,
$\hat A\leq Z(\bL)^F_\ell$, hence $C_{\hat\bG}(\hat A)/Z$ has finite index in
$C_\bG (Z(\bL)^F_\ell)$. So,
$$ C_\bG^\circ(Z(\bL)^F_\ell) \leq (C_{\hat\bG}(\hat A)/Z)^\circ
  = C_{\hat\bG}^\circ(\hat A)/Z =\hat\bL/Z = \bL$$
which proves the claim.

Thus, we may assume that $\bG =[\bG, \bG] $ is simply connected. By
\cite[Lemma~7.1]{KM} and Lemma~\ref{lem:redecusp} we may assume that $\bG$ is
simple. If $\ell $ is good for $\bG$ and odd, then the result is contained in
\cite[Prop.~3.3(ii)]{CE94}. If $\bG $ is of exceptional type and $\ell $ is bad
for $\bG $ then the result is proved case by case in \cite{En00}.
\end{proof}

\subsection{On heights of unipotent characters}
We now collect some results on heights of unipotent characters. We first need
the following observation:

\begin{lem}   \label{lem:posheight}
 Let $\ell$ be a prime and $n\ge\ell$.
 \begin{itemize}
  \item[\rm(a)] The symmetric group $\fS_n$ has an irreducible character of
   degree divisible by $\ell$ unless $n=\ell\in\{2,3\}$.
  \item[\rm(b)] The complex reflection group $G(2e,1,n)\cong C_{2e}\wr\fS_n$
   and its normal subgroup $G(2e,2,n)$ of index~2 (with $e>1$ if $n<4$) have
   an irreducible character of degree divisible by $\ell$.
 \end{itemize}
\end{lem}

\begin{proof}
(a) By the hook formula for the character degrees of $\fS_n$ it suffices to
produce a partition $\la\vdash n$ with no $\ell$-hook, for
$\ell\le n\le 2\ell-1$. For $\ell\ge5$ the partition $(\ell-2,2)\vdash\ell$
and suitable hook partitions for $\ell<n\le 2\ell-1$ are as claimed. For
$\ell\le3$ the symmetric groups $\fS_m$, $\ell+1\le m\le 2\ell$, have suitable
characters.
\par
For (b) note that both $G(2e,1,n)$ and $G(2e,2,n)$ have $\fS_n$ as a factor
group, so we are done by~(a) unless $n=\ell\in\{2,3\}$. In the latter two
cases the claim is easily checked.
\end{proof}

\begin{lem}   \label{lem:unipsmall}
 Let $(\bL,\la)$ be a unipotent $e$-cuspidal pair of $\bG^F$ of central
 $\ell$-defect, where $e=e_\ell(q)$. Suppose that
 $|W_{\bG^F}(\bL, \la)|_\ell\ne 1$ and all irreducible characters of
 $W_{\bG^F}(\bL,\la)$ are of degree prime to $\ell$.
 Then, $\ell\leq 3$. Suppose in addition that $\bG$ is simple and simply
 connected. Then $W_{\bG^F}(\bL,\la)\cong\fS_\ell$ and the following holds:
 \begin{enumerate}
  \item[\rm(a)] If $\ell=3$, then either $\bG^F= \SL_3(q)$ with $3|(q-1)$ or
   $\SU_3(q)$ with $3|(q+1)$ or $\bG$ is of type $E_6$ and $(\bL,\la)$
   corresponds to Line~8 of the $E_6$-tables of \cite[pp.~351,~354]{En00}.
  \item[\rm(b)] If $\ell=2$, then either $\bG$ is of classical type, or
  $\bG$ is of type $E_7$ and $(\bL,\la)$ corresponds to one of Lines~3 or~7 of
  the $E_7$-table of \cite[p.~354]{En00}.
 \end{enumerate}
\end{lem}

\begin{proof}
The first statement easily reduces to the case that $\bG$ is simple, which we
will assume from now on.
We go through the various cases. First assume that $\bG$ is of exceptional
type, or that $\bG^F=\tw3D_4(q)$. The relative Weyl groups $W_{\bG^F}(\bL,\la)$
of unipotent $e$-cuspidal pairs are listed in \cite[Table~1]{BMM}, and
an easy check shows that they possess characters of degree divisible by $\ell$
whenever $\ell$ divides $|W_{\bG^F}(\bL,\la)|$, unless either $\ell=3$, $\bG$
is of type $E_6$ and we are in case~(a), or $\ell=2$ and
$W_{\bG^F}(\bL,\la)\cong C_2$ in $\bG$ of type $E_6$, $E_7$ or $E_8$. According
to the tables in \cite[pp.~351,~354,~358]{En00}, the only
case with $\lambda$ of central $\ell$-defect is in $E_7$ with $\bL$ of type
$E_6$ and $\la$ one of the two cuspidal characters as in~(b).
\par
Next assume that $\bG^F$ is of type $A$. The relative Weyl groups have the
form $C_e\wr\fS_a$ for some $a\ge1$. By definition, $e<\ell$, so if $\ell$
divides $|W_{\bG^F}(\bL,\la)|$ then $\ell\le a$. Then by
Lemma~\ref{lem:posheight} we arrive at either~(a) or~(b) of the
conclusion. If $\bG^F$ is a unitary group, the same argument applies, except
that here the relative Weyl groups have the form $C_d\wr\fS_a$ with
$d =e_\ell(-q)$. For $\bG$ of type $B$ or $C$, the relative Weyl groups have
the form $C_d\wr\fS_a$, with $d\in\{e,2e\}$ even, and again by
Lemma~\ref{lem:posheight} no exceptions arise. The relative Weyl groups have
the same structure for $\bG$ of type $D$, unless $\bG^F$ is untwisted and
$\la$ is parametrised by a degenerate symbol, and either $e\in\{1,2\}$,
$\la=1$, $W_{\bG^F}(\bL,\la)=W$ and so is of type $D_n$ with $n\ge4$, or
$W_{\bG^F}(\bL,\la)\cong G(2d,2,n)$ with $d\ge2$, so again we are done by
Lemma~\ref{lem:posheight}.
\end{proof}

Recall that by \cite[Thm.~A]{En00} if $(\bL,\la)$ is a unipotent $e$-cuspidal
pair of $\bG$, then all irreducible constituents of $R_\bL^\bG(\la)$ lie in
the same $\ell$-block, say $b_{\bG^F}(\bL,\la)$ of $\bG^F$.

\begin{lem}   \label{lem:non-ab-weyl}
 Let $(\bL,\la)$ be a unipotent $e$-cuspidal pair of $\bG^F$ and let
 $B=b_{\bG^F}(\bL,\la)$. Suppose that $\la$ is of central $\ell$-defect
 and that $\bL =C_\bG^\circ (Z(\bL)^F_\ell)$. If $B$ has non-abelian defect
 groups, then $|W_{\bG^F}(\bL,\la)|$ is divisible by $\ell$.
\end{lem}

\begin{proof}
Let $Z= Z(\bL)^F_\ell$ and let $b$ be the block of $\bL^F$ containing $\la$.
Since $\bL=C_\bG^\circ(Z)$, and $Z$ is an $\ell$-subgroup of $\bL$ contained
in a maximal torus of $\bG$, $C_\bG(Z)/\bL$ is an $\ell$-group. Hence,
$\bL^F$ is a normal subgroup of $C_{\bG^F}(Z)$ of $\ell$-power index
and consequently, there is a unique block, say $\tilde b$ of $C_{\bG^F}(Z)$
covering $b$. Further, by \cite[Props.~2.12, 2.13(1), 2.15]{KM} and
\cite[Thm.~3.2]{BMM}, $(Z,\tilde b)$ is a $B$-Brauer pair.

Since $I_{C_{\bG^F}(Z)} (\la) / \bL^F \leq W_{\bG^F}(\bL,\la)$ and since
$C_{\bG^F}(Z)/\bL^F$ is an $\ell$-group, we may assume by way of contradiction
that $ I_{C_{\bG^F}(Z)} (\la) \leq \bL^F $. Further, since $\la$ is of central
$\ell$-defect in $\bL^F$, $\la$ is the unique character of $b$ with $Z$ in
its kernel. Thus, $I_{C_{\bG^F}(Z)}(b) = I_{C_{\bG^F}(Z)}(\la) \leq \bL^F$.
Consequently, $Z$ is a defect group of $\tilde b$.
Now the defect groups of $B$ are non-abelian, whereas $Z$ is abelian. Hence
$N_{\bG^F}(Z, \tilde b)/C_{\bG^F}(Z)$ is not an $\ell'$-group. On the
other hand, $N_{\bG^F}(Z,\tilde b)$ normalises $\bL^F$ and therefore acts by
conjugation on the set of $\ell$-blocks of $\bL^F$ covered by $\tilde b$. Since
$C_{\bG^F}(Z) $ acts transitively on the set of the $\ell$-blocks of $\bL^F$
covered by $\tilde b$, by the Frattini argument,
$N_{\bG^F}(Z, \tilde b) = C_{\bG^F}(Z) N_{\bG^F}(Z, b)$. Hence,
$$N_{\bG^F}(Z, b)/\bL^F= N_{\bG^F}(Z, b)/(N_{\bG^F}(Z, b) \cap C_{\bG^F}(Z))
  \cong N_{\bG^F}(Z, \tilde b)/ C_{\bG^F}(Z)$$
is not an $\ell' $ group. But again since $\la $ is of central $\ell$ defect,
$N_{\bG^F}(Z,b) \leq N_{\bG^F}(\bL,\la)$. Hence  $N_{\bG^F}(\bL, \la)/\bL^F$
is not an $\ell'$ group, contradicting our assumption.
\end{proof}

Recall that by the fundamental result of $e$-Harish-Chandra theory
\cite[Thm.~3.2]{BMM}, for any unipotent $e$-cuspidal pair $(\bL,\la)$ of $\bG$
there is a bijection
$$\rho_{\bL,\la}:\cE(\bG^F,(\bL,\la))\xrightarrow{1-1}\Irr(W_{\bG^F}(\bL,\la))$$
between the set $\cE(\bG^F,(\bL,\la))$ of irreducible constituents of
$R_\bL^\bG(\la)$ and $\Irr(W_{\bG^F}(\bL,\la))$. Moreover we have the following
relationship between the degrees of corresponding characters.

\begin{lem}   \label{lem:weylunipotentdeg}
 Let $(\bL,\la)$ be a unipotent $e$-cuspidal pair of $\bG^F$ and let
 $\chi\in\cE(\bG^F,(\bL,\la))$. Then
 $$\chi(1)_{\ell} = \frac{|\bG^F|_{\ell}\,\la(1)_{\ell}}
   {|\bL^F|_{\ell}\cdot|W_{\bG^F}(\bL,\la)|_{\ell}}\,(\rho_{\bL,\la} (\chi))(1)_{\ell}.$$
 In particular, there exist $\chi_1,\chi_2\in\cE(\bG^F,(\bL,\la))$ with
 $\chi_1(1)_\ell \ne \chi_2(1)_\ell$ if and only if there exists an
 irreducible character of $W_{\bG^F}(\bL,\la)$ with degree divisible by $\ell$.
\end{lem}

\begin{proof}
This follows from \cite[Thm.~4.2 and Cor. 6.3]{MaH}.
\end{proof}


\begin{lem}   \label{lem:unileastheight}
 Let $\bG$ be connected reductive and let $B$ be a unipotent $\ell$-block of
 $\bG^F$. Then $B$ has an irreducible unipotent character of height zero.
\end{lem}

\begin{proof}
We may assume that $\bG=[\bG,\bG]$. Indeed, set $\bG_0=[\bG,\bG]$ and let
$B_0$ be the unipotent block of $\bG_0^F$ covered by $B$. Then the degrees in
$\Irr(B)\cap\cE(\bG^F,1)$ are the same as the degrees in
$\Irr(B_0)\cap\cE(\bG_0^F,1)$. On the other hand, if $\chi\in\Irr(B_0)$
and $\chi'\in\Irr(B)$ covers $\chi$, then $\chi'(1)$ is divisible by $\chi(1)$.
Since every $\chi'\in \Irr(B)$ covers some $\chi\in \Irr(B_0)$ and vice versa
(see for example \cite[Ch.~5, Lemmas~5.7, 5.8]{NT}), we may assume that
$\bG = \bG_0$.

We next claim that we may assume that $\bG$ is simple. Indeed, let
$\bar\bG = \bG/Z(\bG)$ and $\bar B$ the block of $\bar\bG^F$ dominated
by $B$. Let $ H \cong \bG^F/Z(\bG^F)$ be the image of $\bG^F$ in $\bar \bG^F$
under the canonical map from $\bG$ to $\bar \bG$ and let $C$ be the block of
$H$ dominated by $B$. Then $H$ is normal in $\bar \bG^F$ and $C$ is covered
by $\bar B$. The degrees in $\Irr(\bar B) \cap \cE(\bar \bG^F, 1) $ are the
same as the degrees in $\Irr(B) \cap \cE(\bG^F, 1) $ and by the same arguments
as above every irreducible character degree of $\bar B$ is divisible by an
irreducible character degree of $C$ and the set of irreducible character
degrees of $C$ is contained in the set of irreducible character degrees of $B$.
Thus, if the result is true for $B$, it holds for $\bar B$.
So, we may assume that $\bG = [\bG, \bG]$ is simply connected, and hence also
that $\bG$ is simple.

If $\bG$ is of type $A$ and $\ell$ is odd and divides the order of $ Z(\bG^F)$,
then by \cite[Theorem, Prop.~3.3]{CE94} $B$ is the principal block and the
result holds. If $\ell=2$ and $\bG$ is of classical type, then by
\cite[Thm.~13]{CE93} again $B$ is the principal block. In the remaining cases
by the results of \cite{CE94} and \cite{En00} there exists an $e$-cuspidal pair $(\bL,\la)$
for $B$ such that $\la$ is of central $\ell$-defect and a defect group of $B$
is an extension of $Z(\bL^F)_\ell$ by a Sylow $\ell$-subgroup of
$W_{\bG^F}(\bL,\la)$ (see \cite[Thm.~7.12(a) and~(d)]{KM}). Now the result
follows from Lemma~\ref{lem:weylunipotentdeg} by considering the character
in $\cE(\bG^F,(\bL,\la))$ corresponding to the trivial character of
$W_{\bG^F}(\bL,\la)$.
\end{proof}

\begin{lem}   \label{lem:centraldefect}
 Suppose that $\bG$ is simple and let $\la $ be a unipotent $e$-cuspidal
 character of $\bG^F$ of central $\ell$-defect. Then $\la $ is of $\ell$-defect
 zero. Moreover, any diagonal automorphism of $\bG^F$ of $\ell$-power order
 is an inner automorphism of $\bG^F$.
\end{lem}

\begin{proof}
Let $\bG\hookrightarrow\tilde\bG$ be a regular embedding and set $\bar\bG:=
\bG/Z(\bG)$. If $\ell$ is odd, good for $\bG $ and $\ell \ne 3 $ if
$\bG^F= \tw3D_4(q) $, then by
\cite[Prop.~4.3]{CE94}, every unipotent $e$-cuspidal character of
$\bar\bG^F$ and of $\tilde \bG^F$ is of central $\ell$-defect.
The first assertion follows since $\bar\bG^F$ has trivial center and
since $\bar\bG^F$ and $\bG^F$ have the same order. For the second
assertion, note the central $\ell$-defect property of $\la $ as a character of
$\bG^F$ and $\tilde \bG^F$ implies that
$|\tilde \bG^F: Z (\tilde \bG^F)|_\ell = | \bG^F: Z(\bG^F)|_\ell$, hence
$Z(\tilde\bG^F) \bG^F$ is of $\ell'$-index in $\tilde \bG^F$, thus proving
the result.

If $\ell =2$ and $\bG$ is of classical type $A$, $B$, $C$ or $D$ then by
\cite[Thm.~13]{CE93} the principal block of
$\bG^F$ is the only unipotent block of $\bG^F$, and the Sylow $2$-subgroups
of $\bG^F$ are non-abelian, hence $\bG^F$ has no unipotent character of
central $2$-defect. If $\ell$ is bad for $\bG $ and $\bG$ is of exceptional
type, or if $\ell=3 $ and $\bG^F= \tw3D_4(q)$, then the result follows by
inspecting the tables in \cite{En00}. The last assertion follows as in type
$E_6$ the outer diagonal automorphism is of order $3$, but there are no
unipotent $e$-cuspidals of central $3$-defect, and  similarly in type $E_7$,
the outer diagonal automorphism has order $2$, but there are no unipotent
$e$-cuspidals of central $2$-defect.
\end{proof}

\subsection{Some special blocks}
Here we investigate in some detail certain unipotent blocks for $\ell\le3$
related to the exceptions in Lemma~\ref{lem:unipsmall}.

\begin{lem}   \label{lem: SL3}
 Let $\bG^F =\SL_3(q)$, $3|(q-1)$, and let $B$ be the principal 3-block of
 $\bG^F$.
 \begin{enumerate}
  \item[\rm(a)] There exists an irreducible character of positive $3$-height
   in $B$. This contains $Z(\bG^F)$ in its kernel when $q\equiv1\pmod9$.
  \item[\rm(b)] If $q\not\equiv1\pmod9$, then there exists an irreducible
   character in $B$ with $Z(\bG^F)$ in its kernel and which is not stable under
   the outer diagonal automorphism of $\bG^F$.
 \end{enumerate}
 The analogous result holds for $\bG^F=\SU_3(q) $ with $3$ dividing $q+1$.
\end{lem}

\begin{proof}
Let $\bG$ be simple, simply connected of type $A_2$ such that $\bG^F =\SL_3(q)$
with $3|(q-1)$. Then the Sylow 3-subgroups of $\bG^F$ are non-abelian and if
$q\equiv1\pmod9$, then the Sylow $3$-subgroups of $\bG^F/Z(\bG^F)$ are
non-abelian, hence (a) is a consequence of \cite{BE}.
So we may assume that $q\not\equiv1\pmod9$. Let $\eta $ be a primitive third
root of unity in $\FF_q$ and let $t\in\bG^{*F} = \PGL_3(q)$ be the image of
$\diag(1,\eta,\eta^2)$ under the canonical surjection of $\GL_3(q)$
onto $\PGL_3(q)$. So, $C_{\bG^*}^\circ(t)$ is  a maximal torus of $\bG^*$
and $|C_{\bG^*}(t)/C_{\bG^*}^\circ(t)|=3$. Let $\bT$ be an $F$-stable maximal
torus of $\bG$ in duality with $C_{\bG^*}^\circ(t)$ and let $\hat t$ be the
linear character of $\bT^F$ in duality with $t$. Let $\psi$ be an irreducible
constituent of $\RTG(\hat t)$.
Then, $\psi$ is not  stable under the outer diagonal automorphism of $\bG^F$.
Further, $\psi\in\Irr(B)$ as $t$ is a $3$-element and the principal block of
$\bG^F$ is the only unipotent block of $\bG^F$. Finally, $Z(\bG^F)$ is
contained in the kernel of $\psi$ as $t\in [\bG^{*F},\bG^{*F}]$. The proof
for the unitary case is entirely similar.
\end{proof}

\begin{lem}   \label{lem:E6}
 Let $\bG$ be simple, simply connected of type $E_6$, $\bG^F =E_6(q)$,
 $3|(q-1)$, and let $(\bL,\la)$ be a unipotent $1$-cuspidal pair corresponding
 to Line~8 of the $E_6$-table in \cite{En00}.
 \begin{enumerate}
  \item[\rm(a)] There exists an irreducible character of positive $3$-height
   in $B=b_{\bG^F}(\bL,\la)$. This contains $Z(\bG^F)$ in its kernel when
   $q\equiv1\pmod9$.
  \item[\rm(b)] If $q\not\equiv1\pmod9$, then there exists an irreducible
   character in $B$ with $Z(\bG^F)$ in its kernel and which is not stable under
   the outer diagonal automorphism of $\bG^F$.
 \end{enumerate}
 An analogous result holds for $\bG^F=\tw2E_6(q)$ with $3$ dividing $q+1$.
\end{lem}

\begin{proof}
There exists $t\in \bG^{*F}_{3}$ such that $\bM^*:=C_{\bG^*}(t)$ is a $1$-split
Levi subgroup of $\bG^*$ of type $D_5$ containing $\bL^*$, which is contained
in $[\bG^{*F},\bG^{*F}]$ if and only if $q\equiv1\pmod9$, see e.g.
\cite{Lue}. Denoting by $\bM\ge\bL$ an $F$-stable Levi subgroup of $\bG$ in
duality with $\bM^*$ and by $\hat t$ the linear character of $\bM^F$
corresponding to $t$ we thus have that $Z(\bG^F)$ is contained in the kernel
of $\hat t$ if $q\equiv1\pmod9$. Moreover there is an irreducible constituent
$\eta$ of $R_\bL^\bM(\la)$ such that
$\psi:= \epsilon_\bM\epsilon_\bG R_\bM^\bG(\hat t\eta)$ has
$\psi(1)_3>\chi(1)_3$ for any $\chi\in\cE(\bG^F,1)\cap\Irr(B)$. Now
$$ d^{1,\bG^F}(\psi) = \pm d^{1,\bG^F}(R_\bM^\bG(\hat t \eta))
  = \pm R_\bM^\bG (d^{1, \bM^F}(\hat t \eta))
  = \pm R_\bM^\bG (d^{1, \bM^F} (\eta)) = d^{1,\bG^F}(R_\bM^\bG(\eta)). $$
Since $\eta $ is a constituent of $R_\bL^\bM (\la)$ and $\bM$ is $1$-split
in $\bG$, the positivity of $1$-Harish-Chandra theory yields that every
constituent of $R_\bM^\bG (\eta)$ is a constituent of $R_\bL^\bG(\la)$ and
hence in particular $\psi$ is in $\Irr(B)$, proving~(a).
\par
Now assume that $q\not\equiv1\pmod9$. Again by \cite{Lue} there is
$t'\in\bG^{*F}_{3}$ such that $C_{\bG^*}^\circ(t')=\bL^*$, and
$|C_{\bG^*}(t')/C_{\bG^*}^\circ(t')|=3$. Let $\psi'$ be an irreducible
constituent of $R_\bL^\bG(\hat t'\la)$ for $\la\in \cE(\bL^F, 1)$ and
$\hat t $ in duality with $t$. Then $\psi'$ is not stable under the diagonal
automorphism of $\bG^F$, and it lies in $B$ by the same argument as for $\psi$.
The arguments for $\tw2E_6(q)$ are entirely similar.
\end{proof}

\begin{lem}   \label{lem: SL2}
 Let $\bG^F= \SL_2(q)$ with $q$ odd. The principal $2$-block $B$ of $\bG^F$
 contains an irreducible character of even degree. If $q \equiv 1\mod 4$,
 then there exists an irreducible character of even degree in $B$ which
 contains $Z(\bG^F)$ in its kernel. If $q\equiv 3\mod 4$ then there exists an
 irreducible character in $B$ which contains $Z(\bG^F)$ in its kernel and
 which is not stable under the outer diagonal automorphism of $\bG^F$.
\end{lem}

\begin{proof}
 This follows the lines of the proof of Lemma~\ref{lem: SL3}.
\end{proof}

\begin{lem} \label{lem:E7}
 Let $\bG$ be simple, simply connected of type $E_7$, $4|(q-1)$, and let
 $(\bL,\la)$ be a unipotent $1$-cuspidal pair corresponding to Line~3 of the
 $E_7$-table in \cite{En00}.
 \begin{enumerate}
  \item[\rm(a)] There exists an irreducible character of positive $2$-height
   in $B=b_{\bG^F}(\bL,\la)$. This contains $Z(\bG^F)$ in its kernel when
   $q\equiv1\pmod8$.
  \item[\rm(b)] If $q\not\equiv1\pmod8$, then there exists an irreducible
   character in $B$ with $Z(\bG^F)$ in its kernel and which is not stable under
   the outer diagonal automorphism of $\bG^F$.
 \end{enumerate}
 An analogous result holds when $4|(q+1)$ and $(\bL,\la)$ is a unipotent
 $2$-cuspidal pair corresponding to Line~7 of the $E_7$-table in \cite{En00}.
 \end{lem}

\begin{proof}
There exists $t\in \bG^{*F}_2$ of order~4 such that $\bM^*:=C_{\bG^*}(t)$ is
a $1$-split Levi subgroup of $\bG^*$ of type $E_6$ containing $\bL^*$, which
is contained in $[\bG^{*F},\bG^{*F}]$ if and only if $q\equiv1\pmod8$. As in
the proof of Lemma~\ref{lem:E6}, this gives rise to a character as in~(a).
For~(b), consider the involution $t'\in\bL^{*F}$ with
$C_{\bG^*}^\circ(t')=\bL^*$ and $|C_{\bG^*}(t')/C_{\bG^*}^\circ(t')|=2$.
This lies in $[\bG^{*F},\bG^{*F}]$ (see \cite{Lue}), and thus again arguing
as before we find $\psi'\in\Irr(B)$ as in~(b). The arguments for $4|(q+1)$
are entirely similar.
\end{proof}

\subsection{The height zero conjecture for unipotent blocks}
We need the following general observation on covering blocks.

\begin{lem} \label{lem:genred}
 Let $G$ be a finite group, $b$ an $\ell$-block of $G$, $H$ a normal subgroup
 of $G$ and $c$ a block of $H$ covered by $b$.
 \begin{enumerate}
  \item[\rm(a)] Suppose $H$ has $\ell'$-index in $G$. Then a defect group of
   $c$ is a defect group of $b$. Further, $c$ has irreducible character
   degrees with different  $\ell$-heights if and only if $b$ does.
  \item[\rm(b)] Suppose that $H= XY $ where $X$ and $Y$ are commuting normal
   subgroups such that $X\cap Y$ is a central $\ell'$-subgroup of $H$. Let
   $c_X$ be the
   block of $X$ covered by $c$ and let $c_Y$ be the block of $Y$ covered by $c$,
   $D_X$ a defect group of $c_x$ and $D_Y$ a defect group of $c_Y$. Then
   $D_X D_Y$ is a defect group of $c$. In particular, $D$ is non-abelian if
   and only if at least one of $D_X$ or $D_Y$ is non-abelian. Further, $c$ has
   irreducible character degrees with  different  $\ell$-heights if and only
   if one of $c_X$ or $c_Y$ does.
  \item[\rm(c)] Suppose $G =HU$ where $U$ is a central $\ell $-subgroup of
   $G$. Then $b$ has abelian defect groups if and only if $c$ has abelian
   defect groups and $b$ has irreducible characters of different $\ell$-heights
   if and only if $c$ does.
 \end{enumerate}
\end{lem}

\begin{proof}
Part~(a) follows from the Clifford theory of characters and blocks (see for
instance \cite[Ch.~5, Thm.~5.10, Lem.~5.7 and~5.8]{NT}). Part~(b) is
immediate from the fact that $H= X Y$ is a quotient of $X \times Y$ by a
central $\ell'$-subgroup. In~(c), every irreducible character of $H$
extends to a character of $G$, $c$ is $G$-stable and $b$ is the unique block
of $G$ covering $c$, and if $D$ is a defect group of $c$, then $DU$ is a
defect group of $b$.
\end{proof}

\begin{thm}   \label{thm:unibhz2}
 Let $Z$ be a central subgroup of $\bG^F$ and let $\bar B$ be a block of
 $\bG^F /Z$ dominated by a unipotent block $B$ of $\bG^F$. Suppose that
 $\bar B$ has non-abelian defect groups. Then $\bar B$ has irreducible
 characters of different heights.
\end{thm}

\begin{proof}
By Lemma~\ref{lem:unileastheight}, $B$ has a unipotent character of height
zero. Since $Z$ is contained in the kernel of every unipotent character of
$\bG^F$ it suffices to prove that there exists an irreducible character in
$\Irr(B)$ of positive height and containing $Z$ in its kernel.

By \cite[Thm.~A]{En00} there exists a unipotent $e$-cuspidal pair $(\bL,\la)$
of $\bG^F$ such that $B=b_{\bG^F}(\bL,\la)$ with $\la$ of central
$\ell$-defect, unique up to $\bG^F$-conjugacy. Here note that the existence
of such a pair for bad primes is only proved for $\bG$ simple and simply
connected in \cite{En00}, but by Lemma~\ref{lem:centraldefect}, the conclusion
carries over to arbitrary $\bG$.
Suppose first that $\ell \geq 5$. By Lemmas~\ref{lem:owncentraliser}
and~\ref{lem:non-ab-weyl}, $W_{\bG^F}(\bL,\la)$ is not an $\ell'$-group.
Thus, by Lemmas~\ref{lem:weylunipotentdeg} and~\ref{lem:unipsmall}
there are irreducible unipotent characters of different heights in
$\cE(\bG^F, (\bL, \la))$. This proves the claim as $Z$ is in the kernel of
all unipotent characters.

We assume from now on that $\ell\le 3$. Without loss of generality, we may
assume that $ Z$ is an $\ell $-group. We let $\bG$ be a counter-example to
the theorem of minimal semisimple rank. Let $\bX$ be the product of an
$F$-orbit of simple components of $[\bG, \bG] $, and $\bY$ be the product
of the remaining components of $[\bG, \bG]$ (if any) with $Z^\circ (\bG)$.
Then $\bG = \bX \bY$ and $\bX^F \bY^F$ is a normal subgroup of $\bG^F$ of
index $|\bX^F \cap \bY^F| =|Z(\bX^F)\cap Z(\bY^F)|$.
Denote by $B_\bX$ the unique block (also unipotent) of $\bX^F$ covered
by $B$ and let $B_\bY$ be defined similarly. Let $\bar B_\bX$ be the
block of $\bX^FZ/Z \cong \bX^F/(Z\cap \bX^F) $ dominated by $B_\bX$ and
let $\bar B_\bY$ be defined similarly.

Let $\eta \in \Irr(B_\bX)$ with $ Z \cap \bX^F \leq \ker(\eta)$.
We claim that $\eta $ is $\bG^F$-stable and is of height zero in $B_\bX$.
Indeed, let $\tau_\bX \in \Irr (B_\bX) \cap \cE(\bX^F,1)$ and
$\tau_\bY \in \Irr (B_\bY) \cap \cE(\bY^F, 1)$ be of height zero
(see Lemma~\ref{lem:unileastheight}) and let $\tau\in\Irr(B)\cap\cE(\bG^F,1)$
be the unique unipotent
extension of $\tau_\bX\tau_\bY$ to $\bG^F$. Since $Z$ is
central, $\eta $ extends to an irreducible character, say $\hat \eta $ of
$\bX^F Z $ with $Z$ in its kernel. Since $ Z$ is an $\ell$-group, there is
a unique block of $\bX^F Z$ covering $B_\bX $, and this block is
necessarily covered by $B$. Let $\psi $ be an irreducible character of $B$
lying above $\hat \eta $. Then $Z \leq \ker(\psi)$. Any irreducible constituent
of the restriction of $\psi $ to $\bX^F\bY^F$ is of the form $\eta \eta'$,
with $\eta' \in B_\bY$ and
$$ \psi(1) = a |\bG^F : I_{\bG^F}(\eta\eta') | \eta(1) \eta'(1) $$
for some integer $a$ (in fact $a=1$ but we will not use this here).
Since $\psi(1)_\ell =\tau(1)_\ell= \tau_\bX(1) _\ell\tau_\bY (1)_\ell$
and since $ \tau_\bX $ and $\tau_\bY $ are of height zero, it follows
from the above that $\eta $ is of height zero and that
$|\bG^F : I_{\bG^F}(\eta\eta')|$ is not divisible by $\ell $. But
$|\bG^F:I_{\bG^F}(\eta\eta')|$ is divisible by $|\bG^F:I_{\bG^F}(\eta)|$
and the latter index is a power of $\ell$ since $\eta \in \cE_\ell(\bX^F, 1)$.
Thus, $\eta $ is $\bG^F$-stable as claimed.
Similarly, one sees that if $\zeta \in \Irr(B_\bY) $
with $ Z \cap \bY^F \leq \ker(\zeta) $, then $\zeta $ is $\bG^F$-stable and
is of height zero in $B_\bY$. In particular, all elements of
$\Irr(\bar B_\bX)$ and of $\Irr(\bar B_\bY)$ are of height zero.

Suppose that $\ell =3$. By Lemma~\ref{lem:owncentraliser}
and~\ref{lem:non-ab-weyl}, $W_{\bG^F}(\bL,\la)$ has order divisible by $3$.
Thus, by Lemma~\ref{lem:redweyl}, there exists $\bX$ such that
$|W_{\bX^F}(\bL_\bX,\la_\bX)|$ is divisible by $3$ where $(\bL_\bX,\la_\bX)$
is the unipotent $e$-cuspidal pair of $\bX^F$ corresponding to $(\bL,\la)$ by
Lemmas~\ref{lem:redweyl} and~\ref{lem:redecusp}, necessarily of central
$\ell$-defect. By Lemma~\ref{lem:unipsmall},
$W_{\bX^F}(\bL_\bX,\la_\bX)\cong \fS_3$, $|Z(\bX^F)|$ is divisible by $3$
and either the components of $\bX$ are of type $A_2$ or of type $E_6$.
Without loss of generality, we may assume that $\bX$ is simple.
Suppose first that $\bX$ is simple of type $A_2$. By
Lemma~\ref{lem:unipsmall}, $\bX =\bX_{\bf a}$ in the notation
of \cite{CE94}. Hence, by \cite[Thm.~13]{CE93}, $B$ is the principal block of
$B_\bX$. As has been shown above, every irreducible character of $\bX^F$
which contains $\bX^F \cap Z$ in its kernel has height zero and is stable under
$\bG^F$. By Lemma~\ref{lem: SL3} it follows that $Z\cap\bX^F \ne 1$, $3||(q-1)$
(respectively $3||(q+1)$) and that $\bG^F$ induces inner automorphisms of
$\bX^F$, that is $\bG^F =\bX^F \bY^F U$ for some central subgroup $U$ of
$\bG^F$. Since $Z \cap \bX^F \ne 1$, $\bX^F/(Z\cap\bX^F) \cong \PSL_3(q)$
(respectively $\PSU_3(q)$) and $\bX^F/(Z\cap\bX^F)$ is a direct factor of
$\bG^F/Z$. Further, $\bX^F/(Z\cap \bX^F)$ has abelian Sylow $3$-subgroups.
Since $U$ is central in $\bG^F$, it follows by Lemma~\ref{lem:genred} that
the block $\bar B_\bY $ of $\bY^F/(Z\cap \bY^F)$ has non-abelian defect
groups. On the other hand, it has been shown above that all irreducible
characters of $\bar B_\bY$ are of height zero. Hence, $\bY^F/(Z\cap \bY^F)$
is a counter-example to the theorem. But the semisimple rank of $\bY$ is
strictly smaller than that of $\bG$, a contradiction. Exactly the same
argument works for the case that the components of $\bX$ are of type $E_6$
by replacing Lemma~\ref{lem: SL3} with Lemma~\ref{lem:E6}.

Suppose now that $\ell=2 $ and that the components of $\bX $ are of classical
type. Then $\bX^F$ has a unique unipotent $2$-block, namely the principal
block and it follows by the above that all unipotent character degrees of
$\bX^F$ are odd. Thus, the components of $\bX $ are of type $A_1 $, so $\bX^F$
is either $\PGL_2(q^d) $ or $\SL_2(q^d) $ for some $d$. Again we are done by
the same arguments as above using Lemma~\ref{lem: SL2}. Thus, we may assume
that all components of $\bG$ are of exceptional type. By
Lemmas~\ref{lem:owncentraliser} and~\ref{lem:non-ab-weyl}, $W_{\bG^F}(\bL,\la)$
has even order and by Lemma~\ref{lem:redweyl}, there exists $\bX$ such that
$|W_{\bX^F}(\bL_\bX, \la_\bX)|$ is divisible by $2$ where $(\bL_\bX,\la_\bX)$
is the unipotent $e$-cuspidal pair of $ \bX^F$ corresponding to $(\bL, \la) $
necessarily of central $\ell$-defect. Since $\bX$ is of exceptional type,
Lemma~\ref{lem:unipsmall}(b) gives that $\bL_\bX$ is of type
$E_6$ and $\la_{\bX} $ corresponds to either line~3 or~7 of the $E_7$-table of
\cite[p.~354]{En00}. Then  we are done by the same arguments as above using
Lemma~\ref{lem:E7}.
\end{proof}

\subsection{General blocks}
We also need to deal with the so-called quasi-isolated blocks of exceptional
groups of Lie type.

\begin{prop}   \label{prop:quasi}
 Assume that $\bG^F$ is of exceptional Lie type and $\ell$ is a bad prime
 different from the defining characteristic. Let $Z$ be a central subgroup of
 $\bG^F$ and let $\bar B$ be an $\ell$-block of $\bG^F/Z$ dominated by a
 quasi-isolated non-unipotent block $B$ of $\bG^F$. If $\bar B$ has non-abelian
 defect groups, then $\Irr(\bar B)$ contains characters of positive height.
\end{prop}

\begin{proof}
We first deal with the case that $Z=1$, so $\bar B=B$. Here, the quasi-isolated
blocks for bad primes were classified in \cite[Thm.~1.2]{KM}. Any such block
is of the form $B=b_{\bG^F}(\bL,\la)$ for a suitable $e$-cuspidal
pair $(\bL,\la)$ in $\bG$, in such a way that all
constituents of $\RLG(\la)$ lie in $b_{\bG^F}(\bL,\la)$, and the defect groups
are abelian if and only if the relative Weyl group $W_{\bG^F}(\bL,\la)$ has
order prime to~$\ell$.
\par
It is easily checked that all blocks $B$ occurring in the situation of
\cite[Thm.~1.2]{KM} have the following property: either the characters in
$B\cap\cE(\bG^F,\ell')$ lie in at least two different $e$-Harish-Chandra
series, above $e$-cuspidal characters
of different $\ell$-height, or the relative Weyl group has an irreducible
character of positive $\ell$-height. In the first case, the claim follows
since then there are characters in $\Irr(B)\cap\cE(\bG^F,\ell')$ of different
height. In the second case, let $s\in \bG^{*F}$ be a semisimple
(quasi-isolated) $\ell'$-element such that $\Irr(B)\subseteq\cE_\ell(\bG^F,s)$.
Lusztig's Jordan decomposition gives a
height preserving bijection from $\cE(\bG^F,s)$ to the unipotent characters
of the (possibly disconnected) centraliser $\bC=C_{\bG^*}(s)$ of $s$, which
sends $B\cap\cE(\bG^F,s)$ to a collection of $e$-Harish-Chandra series in
$\cE(\bC^F,1)$. As the relative Weyl group has a character of positive
$\ell$-height, a straightforward generalisation of the arguments in
\cite[Cor.~6.6]{MaH} shows that there is an $e$-Harish-Chandra series in
$\cE(\bC^F,1)$ containing characters of different heights, and so there also
exist characters in $B$ of different heights.  \par
Now assume that $Z(\bG^F)\ne1$ and $Z=Z(\bG^F)$, so that $\bG$ is either of
type $E_6$ and $\ell=3$, or of type $E_7$ and $\ell=2$. The only quasi-isolated
block to consider for type $E_6$ is the one numbered~13 in \cite[Tab.~3]{KM},
respectively its Ennola dual in $\tw2E_6$. Since
here the relative Weyl group has characters of positive $3$-height, we get
characters of different height in $\Irr(B)\cap\cE(\bG^F,\ell')$, which have
the centre in their kernel. Similarly, the only cases in $E_7$ are the ones
numbered~1 and~2 in \cite[Tab.~4]{KM}, for which the same argument applies.
\end{proof}

We can now show the Main Theorem for quasi-simple groups of Lie type. Let us
write (BHZ2) for the assertion that blocks with all characters of height
zero have abelian defect groups.

\begin{thm}   \label{thm:nondescbhz2}
 Suppose that $\bG$ is simple and simply connected, not of type $A$, and
 $\ell\ne p$. Then (BHZ2) holds for $\bG^F/Z$ for any central subgroup $Z$
 of $\bG^F$.
\end{thm}

\begin{proof}
We may assume that $Z$ is an $\ell$-group. The Suzuki groups and the Ree
groups $^2G_2(q^2)$ have no non-abelian Sylow subgroups for non-defining
primes. The height zero conjecture for $G_2(q)$, Steinberg's triality groups
$\tw3D_4(q)$ and the Ree groups $\tw2F_4(q^2)$ has been checked in
\cite{Hi90,DM,MaF}. Thus, we will assume that we are not in any of these cases.

Let $B$ be an $\ell$-block of $\bG^F$ and $\bar B$ the $\ell$-block of
$\bG^F/Z$ dominated by $B$. We assume that $\bar B $ has non-abelian defect
groups. Let $ s\in \bG^{*F} $ be a semisimple $\ell'$-element such that
$\Irr(B)\subseteq\cE_\ell(\bG^F,s)$. Let $ \bG_1$ be a minimal $F$-stable Levi
subgroup of $\bG $ such that $C_{\bG^*}(s) \leq \bG_1^*$, thus $s$ is
quasi-isolated in $\bG_1^*$. Let $C$ be a Bonnaf\'e--Rouquier correspondent
of $B$ in $\bG_1^F$, and $\bar C$ the block of $\bG_1^F/Z$ dominated by $C$.
Jordan decomposition induces a defect preserving bijection between
$\Irr(\bar B)$ and $\Irr(\bar C)$ and by \cite[Thm.~1.4]{KM}, $\bar B$ has
abelian defect if and only if $\bar C$ does. Thus it suffices to prove the
result for $C$.
In particular, by Theorem~\ref{thm:unibhz2}, we may assume that $s$ is not
central in $\bG_1$ and hence that $C_{\bG_1^*}(s) =C_{\bG^*}(s)$ is not a
Levi subgroup of $\bG_1^* $ (nor of $\bG^*$).

We first consider the case that $Z(\bG)^F$ is an $\ell'$-group. Let
$\bG\hookrightarrow\tilde\bG$
be a regular embedding. If $\bG$ has connected center we let $\bG=\tilde \bG$.
Let $\tilde B$ be a block of $\tilde\bG^F$ covering $B$ and let
$\tilde s \in \tilde \bG^{*F} $ be a semisimple element such that
$\Irr(\tilde B)\leq\cE(\tilde\bG^F,\tilde s)$. Then by Lemma~\ref{lem:genred}
it suffices to prove that $\tilde B$ has characters of different $\ell$-heights
(note that $Z=1$ here). Further, let $\tilde \bG_1 $ be an $F$-stable Levi
subgroup of $\tilde \bG $ containing $C_{\tilde \bG^*}(\tilde s)$ such that
$\tilde s$ is quasi-isolated in $\tilde\bG_1$ and let $\tilde C$ be a
Bonnaf\'e--Rouquier correspondent of $\tilde B$ in $\tilde \bG_1^F$. By
\cite[Thm.~7.12, Prop.~7.13(b)]{KM}, $\tilde C$ has non-abelian defect groups.
Hence it suffices to prove that $\tilde C$ has irreducible characters of
different $\ell$-heights. By the same reasoning as above, we may assume that
$s$ is not central in $\tilde \bG_1$ and hence that
$C_{\tilde \bG_1^*}(s) =C_{\tilde \bG^*}(s)$ is not a Levi subgroup of
$\tilde \bG_1^*$ (nor of $\tilde\bG^*$).

If moreover $\ell$ is odd and good
for $\tilde \bG_1$, then by \cite{En08}, there is a defect preserving bijection
between $\Irr(\tilde C)$ and $\Irr(C_0)$ for a unipotent block $C_0$ of
$C_{\tilde \bG_1^*}(\tilde s)^F$ whose defect groups are isomorphic to those
of $\tilde C$ and the result follows by Theorem~\ref{thm:unibhz2}.
Enguehard has informed us that the
prime $3$ should have been excluded from the results of \cite{En08}. However,
for classical groups with connected center Jordan decomposition commutes with
Lusztig induction (see for instance appendix to latest version of \cite{En08})
and hence by \cite[Thm.~2.5]{CE93b} and \cite[5.1,~5.2]{CE99} the prime $3$
may be included in the above.

Thus, we may assume that if $\ell$ is odd and $Z(\bG)$ is an $\ell'$-group,
then $\ell$ is bad for $\tilde \bG_1$ and hence for $\tilde \bG$ and $\bG$.
We now consider the various cases. Suppose that
$\bG$ is classical of type $B,C,D$. If $\ell=2$, then $s$ has odd order and
$C_{\bG^*}(s)$ is a Levi subgroup of $\bG^*$, a contradiction. If $\ell$ is
odd, then $\ell$ is good for $\bG$. On the other hand, $Z(\bG)$ is a $2$-group,
a contradiction.

So, $\bG$ is of exceptional type. If $\ell$ is good for $\bG$, then
$\ell\geq 5$, and in all cases $Z(\bG)$ is an $\ell'$-group, a contradiction.
Thus $\ell$ is bad for $\bG$. Then by Proposition~\ref{prop:quasi}, $\bG_1$ is
proper in $\bG$. Suppose that $\ell=5$ and so $\bG$ is of type $E_8$. Since
$Z(\bG)=1$, $5$ is bad for $\bG_1$. Thus $\bG=\bG_1$, a contradiction.

Now assume that $\ell=3$. Suppose that $\bG$ is of type $F_4$. Then all
components of $[\bG_1,\bG_1]$ are classical, hence $3$ is good for $\bG_1$
and $Z(\bG)$ is connected, a contradiction.

Suppose $\bG$ is of type $E_6$. If all components of $\bG_1$ are of type $A$,
then $C_{\bG_1^*}^\circ(s)$ is a Levi subgroup of $\bG_1$. On the other hand,
$Z(\bG_1)/Z^\circ(\bG_1)\leq Z(\bG)/Z^\circ(\bG)$ is a $3$-group, and $s$ is
a $3'$-element, hence $C_{\bG_1^*}(s)$ is connected. So, $C_{\bG_1^*}(s)$ is
a Levi subgroup of $\bG_1^*$, a contradiction.
Suppose $\bG_1$ has a component, say $\bH$ of type $D_4$ or $D_5$.
So $\bG_1= \bH Z^\circ(\bG_1)$. Since the centre of $\bH$ is a $2$-group, by
Lemma~\ref{lem:genred} we may replace $\bG_1^F/Z$ with the direct product of
$\bH^F$ and $Z^\circ(\bG_1)/Z$. Since (BHZ2) has been shown to be true for
$\bH^F$ above (here note that $\bH$ is simply-connected), $\bH^F$ has abelian
Sylow $3$-subgroups and we are done.

Suppose $\bG$ is of type $E_7$. Then $|Z(\bG)|=2 $, hence $3$ is bad for
$\tilde \bG_1$ and it follows that $[\bG_1, \bG_1] $ is of type $E_6$ (note
that if $\bG_1$ is proper in $\bG$ then $\tilde \bG_1$ is proper in
$\tilde\bG$). Denoting by $\bar s$ the image of $s$ in $[\bG_1, \bG_1]^*$
and by $D$ a block of $[\bG_1, \bG_1]^F$ covered by $C$, one sees that $D$
corresponds to one of the lines 13, 14, 15 of Table~3 of \cite{KM}. If $D$
corresponds to one of the lines 13 or 14, there are irreducible characters of
different $3$-heights in $\cE ([\bG_1,\bG_1]^F,\bar s) \cap \Irr(D)$.
But since $\bG_1$ has connected centre, and since
$Z([\bG_1, \bG_1])/Z^\circ([\bG_1, \bG_1])$ is a $3$-group and $s$ has order
prime to $3$, all characters in $\cE([\bG_1,\bG_1]^F,\bar s)$ are
$\bG_1^F$-stable
and extend to irreducible characters of $\bG_1^F$ (see \cite[Cor.~11.13]{B06}).
All irreducible characters of $\bG_1^F$ covering the same irreducible character
of $[\bG_1,\bG_1]^F$ have the same degree and every element of $\Irr(D)$ is
covered by an element of $\cE(\bG_1^F,s)\cap\Irr(C)$. Thus there exist
elements in $\Irr(C)\cap\cE(\bG_1^F,s)$ of different $3$-heights. If $D$
corresponds to line 15, then $3$ does not divide the order of $Z(\bG_1^F)$.
Hence, $\bG_1^F=Z^\circ(\bG_1^F)\times[\bG_1,\bG_1]^F$.
By \cite[Prop.~4.3]{KM}, $D$ has abelian defect groups hence so does $C$
and there is nothing to prove.

If $\bG$ is of type $E_8 $, then exactly the same arguments as in the $E_7$
case apply hence we are left with one of the following cases:
$[\bG_1,\bG_1]$ is of type $E_6+A_1$ or of type $E_7$. In the former case,
by Lemma~\ref{lem:genred} we may assume that the fixed point subgroup of the
component of type $A_1 $ is a direct factor of $\bG_1^F$ and so has abelian
Sylow $3$-subgroups. Therefore, we may assume that $[\bG_1,\bG_1]$ is of
type $E_6$ and we are done by the same argument as in the case that $\bG$ is
of type $E_7$. If $[\bG_1,\bG_1]$ has type $E_7$, then
$$|\bG_1^F:[\bG_1,\bG_1]^F Z^\circ (\bG_1)^F|
  = |[\bG_1,\bG_1]^F\cap Z^\circ(\bG_1)^F|=2,$$
hence by Lemma~\ref{lem:genred} we may assume that $\bG_1$ is simple of type
$E_7$, and we are done by Proposition~\ref{prop:quasi}.

Finally suppose that $\ell=2$. In case $\bG$ is of type $E_6$, we may replace
$\bG$ by
$\tilde\bG$ by Lemma~\ref{lem:genred} and still keep the assumption that
$\tilde\bG_1$ is proper in $\tilde\bG$. Thus, either $Z(\bG)$ is connected
or $Z(\bG)/Z^\circ(\bG)$ has order $2$ (in case $\bG$ is of type $E_7$).
Consequently, since $s$ has odd order, $C_{\bG_1^*}(s) =C_{\bG^*}(s)$ is
connected. Thus, if all components of $[\bG_1,\bG_1]$ are of classical type,
then $C_{\bG_1^*}(s)$ is a Levi subgroup of $\bG_1^*$, a contradiction.
We are left with the following cases: $\bG$ is of type $E_7$ and
$[\bG_1,\bG_1]$ is of type $E_6$, or $\bG$ is of type $E_8$ and $[\bG_1,\bG_1]$
is of type $E_6$, $E_6+A_1$ or $E_7$.

Suppose that $[\bG_1,\bG_1]$ is of type $E_6$. Since $C_{\bG_1^*}(s)$ is
connected and $s$ is quasi-isolated in $\bG_1^*$, $C_{\bG_1^*}^\circ(s)$ has
the same semisimple rank as $\bG_1^*$. Thus, $\bar s$ and $D$ correspond to one
of the lines 1, 2, 6, 7, 8 or 12 of Table~3 of \cite{KM}. In all of these
cases, there are characters in $\cE([\bG_1,\bG_1]^F,\bar s)\cap\Irr(D)$ of
different $2$-heights. Since $Z(\bG)/Z^\circ(\bG)$ is a $2$-group, every
element of $\cE([\bG_1,\bG_1]^F,\bar s)\cap\Irr(D)$ extends to an element of
$\Irr(C)\cap\cE(\bG_1^F, s)$. Since $Z$ is in the kernel of all characters
in $\cE(\bG_1^F,s)$, $\bar B$ has characters of different $2$-heights and
we are done.

Suppose $\bG $ is of type $E_8 $ and $[\bG_1,\bG_1]$ is of type $E_6+A_1$.
Then by Lemma~\ref{lem:genred}, we may assume that
$\bG_1^F =\bH_1^F \times \bH_2 ^F$, where $\bH_1^F$ is isomorphic to
$E_6(q)$ or $\tw2E_6(q)$, $\bH_2$ has connected center and $[\bH_2, \bH_2]$
has a single component of type $A_1$. Since the block of $\bH_2^F$ covered
by $C$ is quasi-isolated, we may assume that $C$ covers a unipotent (in fact
the principal) block of $\bH_2^F$. If $\bH_2^F/Z$ has non-abelian Sylow
$2$-subgroups, then we are done by Theorem~\ref{thm:unibhz2}. If the block
of $\bH_1^F$ covered by $C$ has non-abelian defect groups, then we are done
by Proposition~\ref{prop:quasi}.

Finally, assume that $\bG$ is of type $E_8 $ and $[\bG_1,\bG_1]$ is of type
$E_7$. Since $s$ is not central in $\bG_1$, $1\ne \bar s$ is a quasi-isolated
element of $[\bG_1,\bG_1]^*$. By Table~5 of \cite{KM} the block $D$ of
$[\bG_1, \bG_1]^F$ has non-abelian defect groups. Now we are done by
the same argument as given at the end of Proposition~\ref{prop:quasi}.
\end{proof}

\section{Brauer's height zero conjecture for quasi-simple groups } \label{sec:height0}

\begin{proof}[Proof of the Main Theorem]
We invoke the classification of finite simple groups. One direction of the
assertion has been shown in \cite[Thm.~1.1]{KM}. So we may now assume that
all $\chi\in\Irr(B)$ have height zero. We need to show that $B$ has abelian
defect groups. If $S$ is a covering group of a
sporadic simple group or of $\tw2F_4(2)'$ it can be checked using the tables
in \cite{Atl} that the only $\ell$-blocks with defect groups of order at
least $\ell^3$ and all characters in $\Irr(B)$ of height zero are the
principal 2-block of $J_1$, the principal 3-block of $O'N$ and a 2-block of
$Co_3$ with defect groups of order $2^7$. For the first two groups, Sylow
$\ell$-subgroups are abelian, and the latter block has elementary abelian
defect groups, see \cite[\S7]{La78}.
\par
Similarly, if $S$ is an exceptional covering group of a finite simple group of
Lie type, again by \cite{Atl} there is no such block of positive defect at all.
\par
The height zero conjecture for alternating groups $\fA_n$, $n\ge7$, and their
covering groups was verified in \cite{Ol93}, for example, except for the
2-blocks of the double covering $2.\fA_n$. Since the height zero conjecture
has been checked for the 2-blocks of $\fA_n$ we know that the only 2-blocks
of $2.\fA_n$ which could possibly consist of characters of height zero are
those whose defect groups in $\fA_n$ are abelian. But the latter have defect
group of order at most~4, so the defect groups in $2.\fA_n$ have order at
most~8, and for those the claim is again known by work of Olsson \cite{Ol75}.
\par
Now assume that $S$ is of Lie type. If $\ell$ is the defining characteristic
of $S$, then the result is contained in Proposition~\ref{prop:defchar}. We
may hence suppose that $\ell$ is a non-defining prime. There, Brauer's height
zero conjecture for groups of type $A_n$ has been shown by Blau and Ellers
\cite{BE}. For all the other types, the claim is shown in
Theorem~\ref{thm:nondescbhz2}.
\end{proof}



\begin{thebibliography}{131}

\bibitem{BE}
{\sc H. Blau, H. Ellers}, Brauer's height zero conjecture for central
 quotients of special linear and special unitary groups. J. Algebra {\bf212}
  (1999), 591--612.

\bibitem{B06}
{\sc C. Bonnaf\'e}, \emph{Sur les Caract{\`e}res des Groupes R\'eductifs
  Finis a Centre non Connexe : Applications aux Groupes Sp\'eciaux Lin\'eaires
  et Unitaires}. Ast{\'e}risque {\bf306} (2006).

\bibitem{BMM}
{\sc M. Brou\'e, G. Malle, J. Michel}, Generic blocks of finite reductive
  groups. Ast\'erisque  No. 212 (1993), 7--92.

\bibitem{CE93}
{\sc M. Cabanes, M. Enguehard}, Unipotent blocks of finite reductive groups
  of a given type. Math. Z. {\bf 213} (1993), 479--490.

\bibitem {CE93b}
{\sc M. Cabanes, M. Enguehard}, On general blocks of finite reductive groups:
  ordinary characters and defect groups. Rapport de Recherche du LMENS 93-13
  (1993).

\bibitem{CE94}
{\sc M. Cabanes, M. Enguehard}, On unipotent blocks and their ordinary
  characters. Invent. Math. {\bf 117} (1994), 149--164.

\bibitem{CE99}
{\sc M. Cabanes, M. Enguehard}, On blocks of finite reductive groups and
  twisted induction.  Adv. Math. {\bf145} (1999), 189--229.

\bibitem{Atl}
{\sc J.H.~Conway, R.T.~Curtis, S.P.~Norton, R.A.~Parker, R.A.~Wilson},
  \emph{Atlas of Finite Groups}. Clarendon Press, Oxford, 1985.

\bibitem{DM}
{\sc D. Deriziotis, G. Michler}, Character table and blocks of finite simple
  triality groups $\tw3D_4(q)$. Trans. Amer. Math. Soc. {\bf303} (1987), 39--70.

\bibitem{En00}
{\sc M. Enguehard}, Sur les $l$-blocs unipotents des groupes r\'eductifs
  finis quand $l$ est mauvais. J. Algebra {\bf230} (2000), 334--377.

\bibitem{En08}
{\sc M. Enguehard}, Vers une d\'ecomposition de Jordan des blocs des groupes
  r\'eductifs finis. J. Algebra {\bf319} (2008), 1035--1115.

\bibitem{Hi90}
{\sc G. Hiss}, \emph{Zerlegungszahlen endlicher Gruppen vom Lie-Typ in
  nicht-definierender Charakteristik}. Habilitationsschrift, RWTH Aachen,
  1990.

\bibitem{Hum}
{\sc J. E. Humphreys}, Defect groups for finite groups of Lie type.
  Math. Z. {\bf119} (1971), 149--152.

\bibitem{KM}
{\sc R. Kessar, G. Malle}, Quasi-isolated blocks and Brauer's height
  zero conjecture. Annals of Math. {\bf178} (2013), 321--384.

\bibitem{La78}
{\sc P. Landrock}, The non-principal $2$-blocks of sporadic simple groups.
  Comm. Algebra {\bf6} (1978), 1865--1891.

\bibitem{Lue}
{\sc F. L\"ubeck}, Table at http://www.math.rwth-aachen.de/~Frank.Luebeck/chev/index.html

\bibitem{Lu88}
{\sc G. Lusztig}, On the representations of reductive groups with disconnected
  centre. Ast\'erisque {\bf168} (1988), 157--166.

\bibitem{MaF}
{\sc G. Malle}, Die unipotenten Charaktere von $\tw2F_4(q^2)$. Comm. Algebra
  {\bf18} (1990), 2361--2381.

\bibitem{MaH}
{\sc G. Malle}, Height~0 characters of finite groups of Lie type.
  Represent. Theory {\bf11} (2007), 192--220.

\bibitem{MT}
{\sc G. Malle, D. Testerman}, \emph{Linear Algebraic Groups and Finite
  Groups of Lie Type}. Cambridge Studies in Advanced Mathematics, 133,
  Cambridge University Press, Cambridge, 2011.

\bibitem{NS14}
{\sc G. Navarro, B. Sp\"ath}, On Brauer's height zero conjecture. J. Eur.
  Math. Soc. {\bf16} (2014), 695--747.

\bibitem{Ol75}
{\sc J. B. Olsson}, On $2$-blocks with quaternion and quasidihedral defect
  groups. J. Algebra {\bf36} (1975), 212--241.

\bibitem{Ol93}
{\sc J. B. Olsson}, \emph{Combinatorics and representations of finite groups}.
  Universit\"at Essen, Fachbereich Mathematik, Essen, 1993.

\bibitem{NT}
{\sc H. Nagao, Y. Tsushima}, \emph{Representations of Finite Groups}.
  Academic Press,  Boston, 1989.


\end{thebibliography}
\end{document}